\documentclass[12pt]{article}
\usepackage[usenames]{color}
\usepackage{tabularx,colortbl}
\usepackage{amsmath,amssymb,amsthm,amsfonts,mathtools}
\usepackage{setspace}

\newtheorem{thm}{Theorem}[section]
\newtheorem{prop}{Proposition}[section]
\newtheorem{lem}{Lemma}[section]

\definecolor{MyDarkBlue}{rgb}{0,0.08,0.45}

\definecolor{ColorName}{rgb}{0,0.6,0}

\usepackage[section]{algorithm}
\usepackage{algpseudocode,algorithmicx}
\newcommand*\Let[2]{\State #1 $\gets$ #2}

\usepackage{soul}

\begin{document}
\pagestyle{plain} 
\pagenumbering{arabic}

\title{Qin's Algorithm, Continued Fractions and 2-dimensional Lattices}

\author{Han Wu\thanks{School of Cyber Science and Technology, Shandong University, Qingdao {\rm  266237}, China;  e-mail: {\tt hanwu97@mail.sdu.edu.cn}.}
~and Guangwu Xu\thanks{School of Cyber Science and Technology, Shandong University, Qingdao {\rm  266237}, China; e-mail: {\tt gxu4sdq@sdu.edu.cn}. (Corresponding author)  }}

\maketitle

\begin{abstract}
In his ``Mathematical Treatise in Nine Sections''(1247), Jiushao Qin introduced ``the Dayan method of deriving unity'' for computing modular inverses, an essential component 
of his treatment of the Chinese remainder theorem. 
A distinctive feature of Qin's algorithm is its use of a state consisting of four nonnegative integers arranged in a matrix. 
In this paper, we investigate the mathematical structure of these states and their connections with continued fractions and two-dimensional lattices. 
By introducing a signed version of the state, we obtain a natural matrix representation that reveals the convergents of continued fractions and gives direct derivations of several classical results. 
We then show that the signed states form bases of a family of two-dimensional lattices arising from modular congruences. 
In particular, we prove that a shortest lattice vector can be obtained from the rows of a signed state or their sum or difference. 
We also study the monotonicity of the inner products of the state vectors and discuss a subsequent result that locates a shortest vector in one of two consecutive states. It is
surprising that a shortest lattice vector is found in a highly structure sequence of states.
These results demonstrate that the state structure of Qin's algorithm provides a unified framework for understanding modular inversion, continued fractions, and shortest vectors in two-dimensional lattices.


 MSC(2020) 11A55, 11Y16, 52C05
\end{abstract}

\setcounter{footnote}{0}

\section{Introduction}
In his 1247 book ``Mathematical Treatise in Nine Sections''\cite{Qin}, Jiushao Qin introduced  ``the Dayan method of aggregation''(Dayan zongshu shu), which contains a detailed and general treatment of the Chinese remainder theorem (CRT). One of the key computational components of this method is modular inversion. Qin described an algorithm for computing modular inverses, which he called `` the Dayan method of derving unity'' (Dayan qiuyi shu)\footnote{Here, ``deriving unity'' refers to the constructive process by which Qin's algorithm obtains the unit \(1\), rather than merely searching for it.}.
 Faithful modern interpretations of Qin's algorithm have been studied in \cite{Xu,XuLi}, where several distinctive features of the algorithm, particularly those differing from the extended Euclidean algorithm, were identified and analyzed.

Given coprime positive integers $1<a<m$, Qin's algorithm computes the least positive inverse of $a$ modulo $m$. The following is an English translation of Qin's original procedure, taken from \cite{Libbrect}:

\begin{center}
	\begin{tabular}{|p{12cm}|}
		\hline
		${}\quad $ {\bf Qin's Algorithm: DaYan Deriving One} \\
		\hline
		$\bullet$ {\tt Set up the number $a$ at the right hand above, the number $m$ at the right hand below.
			Set  $1$ at the left hand above.}\\
		$\bullet$ {\tt First divide the `right below' by the `right above', and the quotient
			obtained, multiply it by the $1$ of  `left above' and add it to  `left below'}.\\
		$\bullet$ {\tt  After this, in the `upper' and `lower' of the right column,
			divide  the larger number by the smaller one. Transmit  and divide them by each other.
			Next bring over the quotient obtained and [cross-] multiply with each other. Add the `upper' and the `lower' of the
			left column.}\\
		$\bullet$ {\tt One has to go until the last remainder of the `right above' is $1$
			and then one can stop. Then  you examine the result of `left above'; take it as the modular inverse.}
		\\
		\hline
	\end{tabular}
\end{center}

This ancient procedure is remarkably close to a modern algorithm. One of its most distinctive features is the use of four variables arranged in the following two-by-two matrix:
\[
\begin{matrix}
	\mbox{\tt left-above} & \mbox{\tt right-above}\\
	\mbox{\tt left-below} & \mbox{\tt right-below}
\end{matrix}.
\]
We shall call this matrix a {\sl state} of the algorithm. 
Starting from the initial state
\[
\begin{pmatrix}
	1&a\\
	0&m
\end{pmatrix},
\]
Qin's procedure repeatedly updates the entries through elementary operations until the entry {\tt right-above} becomes $1$. The entry {\tt left-above}  then gives the desired modular inverse $a^{-1}$. Here and throughout the paper,  $a^{-1}$ denotes the least positive inverse of $a$ modulo $m$. 

As discussed later, the entries {\tt right-above} and {\tt right-below} contain the remainders arising from successive divisions. This is one of several features that distinguish Qin's algorithm from the extended Euclidean algorithm presented in \cite{Bach}.
Unlike the usual extended Euclidean algorithm, all entries of Qin's states remain nonnegative throughout the computation.

The final state of Qin's algorithm is of the form
$\begin{pmatrix}
	a^{-1} &1\\
	*&*
\end{pmatrix}$.
As we shall see, one further step, consisting of a simple row transformation, yields the state
\[
\begin{pmatrix}
	a^{-1}&1\\
	m&0
\end{pmatrix}.
\]
Thus, the algorithm reveals a remarkable duality between the initial value $a$ and its modular inverse $a^{-1}$, reflected directly in the structure of the state matrices.

The arrangement of the four variables is not merely a convenient computational device. It has a natural mathematical interpretation through matrix operations and satisfies a simple but powerful invariant. More importantly, the information encoded in the states extends well beyond modular inversion. The purpose of this paper is to investigate this underlying structure and show how it connects Qin's algorithm with two important areas of mathematics: continued fractions and shortest vectors in two-dimensional lattices.

We  introduce a signed version of Qin's state, called the {\sl s-state}, obtained by changing the sign of the upper-right entry. This simple modification makes the matrix structure particularly transparent: successive s-states are related by unimodular row transformations, and their determinants remain equal to $m$. This determinant identity, which we refer to as {\sl Qin's invariant}, plays a central role throughout the paper.

Our first main connection is with continued fractions. We establish an explicit matrix relation between the s-states of Qin's algorithm and the convergents of the continued fraction expansion of $\dfrac{a}m$. In particular, the entries of the s-states encode both the denominators of consecutive convergents and their approximation errors. This interpretation gives direct derivations of several classical properties of continued fractions, including the determinant identity for consecutive convergents, their alternating approximation behavior, and standard error estimates. It also reveals a simple relation between the penultimate convergent and the modular inverse computed by Qin's algorithm.

Our second main connection concerns the two-dimensional lattice
\[
\Lambda(a,m)=
\{(x,y)\in\mathbb Z^2: ax+y\equiv0\pmod m\}.
\]
We show that every s-state generated by Qin's algorithm is a basis matrix of this lattice. Thus, Qin's procedure naturally produces a sequence of lattice bases through unimodular transformations. More significantly, we prove that a shortest vector of $\Lambda(a,m)$ can always be obtained from the two row vectors of some s-state, either directly or by taking their sum or difference. This establishes a connection between Qin's algorithm and the shortest-vector problem in this family of lattices. 
The result is rather surprising. Although $\Lambda(a,m)$ has infinitely many bases, Qin's algorithm 
generates only a small, highly structured sequence of states. Nevertheless, a shortest vector of
$\Lambda(a,m)$ is guaranteed to be obtained from one of the states.

To locate a shortest vector more precisely, we study the inner products of the two row vectors of the s-states and show that they form a strictly increasing sequence. In an earlier version of this work \cite{WuXu}, we conjectured that a shortest vector must occur as a row of one of the two consecutive s-states at which the inner product changes sign. This conjecture was subsequently proved by Zou \cite{Zou}. A key ingredient in Zou's proof is our shortest-vector theorem, which restricts the candidates to the two rows of an s-state and their sum or difference. Zou's result eliminates the latter two possibilities and further identifies the two consecutive states containing a shortest vector. Thus, the two results together provide a particularly simple characterization and an effective method for locating shortest vectors in this family of lattices.

Taken together, these results show that the state structure of Qin's algorithm is mathematically much richer than its original purpose of computing modular inverses. The same sequence of elementary operations simultaneously encodes continued fraction approximations and produces lattice bases from which shortest vectors can be identified. In this sense, Qin's algorithm provides a natural link between classical number theory and the geometry of numbers.

The remainder of the paper is organized as follows. Section \ref{sec:QinAlg} presents a modern formulation of Qin's algorithm and develops the basic properties of its states and s-states. Section \ref{sec:CF} establishes their connection with continued fractions and derives several classical consequences. Section \ref{sec:Latt} studies the associated two-dimensional lattices, proves the shortest-vector result, and discusses the monotonicity of inner products and Zou's theorem.

\section{The Dayan Method of  Driving Unity and Its Properties}\label{sec:QinAlg}

\subsection{The Formulation of Qin's Dayan Method of  Driving Unity}

We use
\[
\begin{pmatrix}
	x_{11}&x_{12}\\
	x_{21}&x_{22}
\end{pmatrix}
\triangleq
\begin{pmatrix}
	\mbox{\tt left-above}&\mbox{\tt right-above}\\
	\mbox{\tt left-below}&\mbox{\tt right-below}
\end{pmatrix}
\]
to denote the state in Qin's Dayan method of driving unity, so that the method can be expressed in modern pseudocode. The initial state is
\[
\begin{pmatrix}
	x_{11}&x_{12}\\
	x_{21}&x_{22}
\end{pmatrix}
=
\begin{pmatrix}
	1&a\\
	0&m
\end{pmatrix},
\]
while the final state is of the form
\[
\begin{pmatrix}
	x_{11}&x_{12}\\
	x_{21}&x_{22}
\end{pmatrix}
=
\begin{pmatrix}
	a^{-1}&1\\
	*&*
\end{pmatrix},
\]
where $a^{-1}$ denotes the least positive residue of the inverse of $a$ modulo $m$.

We first remark on the termination condition ``{\tt until the last remainder of the `right above' is $1$},'' that is, $x_{12}=1$ in the final state. This condition has been questioned in several papers appearing in \cite{Wu_ed} (see also \cite{Libbrect}). Such a concern is indeed justified if one uses the usual division of positive integers,
\[
c=\bigg\lfloor\frac{c}{d}\bigg\rfloor d+r,
\]
where the remainder $r$ is the least nonnegative residue modulo $d$, i.e., $0\le r<d$.

However, we believe that Qin made no mistake in his termination condition: after an even number of steps, $x_{12}=1$ can always be achieved. The fact that the number of steps is even is itself an interesting feature of Qin's design. The key observation is that the division should instead use the least positive residue modulo $d$ as the remainder. This form of division is also mentioned in \cite{Wang}, and a detailed explanation was given in \cite{XuLi}.

We briefly recall the historical motivation. In ancient China, a form of division in which the remainder is required to be the least positive residue modulo the divisor appears to have been used. For example, a traditional divination method associated with the ``I Ching'' (Book of Changes, 1000--400 BC) generates a hexagram by manipulating $50$ yarrow stalks. In this process, division by $4$ is used, with the remainder required to belong to $\{1,2,3,4\}$. Qin also described this divination method in his book \cite{Qin}. Such a division can be expressed as follows: for positive integers $c$ and $d$, there is a unique $r$ with $1\le r\le d$ such that
\[
c=\bigg\lfloor\frac{c-1}{d}\bigg\rfloor d+r.
\]
Thus, $r$ is the least positive residue of $c$ modulo $d$.

Using this form of division, we can formulate Qin's algorithm in modern language while remaining faithful to his original procedure. In particular, the termination condition $x_{12}=1$ can always be achieved \cite{XuLi}.

\begin{algorithm}[htb]
	\caption{Qin's Algorithm: Dayan Method of Deriving Unity
		\label{alg-Qin}}
	\begin{algorithmic}[1]
		\Require{Integers $a,m$ with $1<a<m$, $\gcd(a,m)=1$}
		\Ensure{ A positive integer $u$ such that $ua\equiv1\pmod m$.}
		\Function{Qin}{$a, m$}
		\Let{$\begin{pmatrix}x_{11}&x_{12}\\x_{21}&x_{22}\end{pmatrix}$}{$\begin{pmatrix}1&a\\0&m\end{pmatrix}$ }
		\While{($x_{12}\neq1$)}
		\If{($x_{22}>x_{12}$)}
		\Let{$q$}{$\big\lfloor \dfrac{x_{22}-1}{x_{12}} \big\rfloor$ }
		\Let{$x_{21}$}{$x_{21}+qx_{11}$}
			\Let{$x_{22}$}{$x_{22}-qx_{12}$ \quad (*the remainder*)}
	\ElsIf{($x_{12}>x_{22}$)}
			\Let{$q$}{$\big\lfloor \dfrac{x_{12}-1}{x_{22}} \big\rfloor$ }
%
			\Let{$x_{11}$}{$x_{11}+qx_{21}$}
				\Let{$x_{12}$}{$x_{12}-qx_{22}$ \quad (*the remainder*)}
			\EndIf
		\EndWhile
		\State \Return{$x_{11}$}
		\EndFunction
	\end{algorithmic}
\end{algorithm}

We next explain in more detail why the least positive residue, rather than the least nonnegative residue, should be used in Qin's algorithm.

The first step updates the second row of the state, the second step updates the first row, and the two types of updates continue alternately. Since the algorithm terminates only when an update of the first row produces $x_{12}=1$, termination must occur at an even-numbered step. This observation was made in \cite{XuLi,Xu}.

For fixed integers $a,m$ with $1<a<m$ and $\gcd(a,m)=1$, denote the state after step $k$ of Qin's algorithm by
\[
{\cal X}_k=
\begin{pmatrix}
	x_{11}^{(k)}&x_{12}^{(k)}\\
	x_{21}^{(k)}&x_{22}^{(k)}
\end{pmatrix}.
\]
We also denote the quotient used at step $k$ by $q_k$.

Using the least nonnegative residues, the Euclidean algorithm gives
\[
\begin{array}{lcl}
	m&=&\bar q_1a+r_1,\\
	a&=&\bar q_2r_1+r_2,\\
	r_1&=&\bar q_3r_2+r_3,\\
	&&\quad\cdots\\
	r_{n-3}&=&\bar q_{n-1}r_{n-2}+r_{n-1},\\
	r_{n-2}&=&\bar q_nr_{n-1}+r_n,
\end{array}
\]
where
\[
1=r_n<r_{n-1}<\cdots<r_1<a<m.
\]

For $k<n$, since $1<r_k<r_{k-1}$,
\[
\bigg\lfloor\frac{r_{k-1}-1}{r_k}\bigg\rfloor
=
\bigg\lfloor\frac{r_{k-1}}{r_k}\bigg\rfloor,
\]
and hence $\bar q_k=q_k$.

If $n$ is even, the last step of Qin's algorithm gives
\[
x_{12}^{(n)}=r_n=1.
\]
In this case, we also have $\bar q_n=q_n$.

If $n$ is odd, then at step $n$ we have
\[
x_{22}^{(n)}=r_n=1,
\]
but $x_{12}=1$ has not yet been reached. Qin's procedure therefore performs one additional step  with
\[
q_{n+1}
=
\bigg\lfloor\frac{r_{n-1}-1}{r_n}\bigg\rfloor
=
r_{n-1}-1.
\]
Consequently,
\[
x_{12}^{(n+1)}
=
x_{12}^{(n)}-q_{n+1}x_{22}^{(n)}
=
r_{n-1}-(r_{n-1}-1)\cdot1
=
1.
\]
Thus, in this case,
\[
\bar q_n=q_n,
\qquad
q_{n+1}=r_{n-1}-1=x_{12}^{(n)}-1.
\]
This is precisely the situation in which the least positive residue is essential to Qin's procedure. At all earlier steps, using the least positive residue has the same effect as using the least nonnegative residue.

\subsection{Properties of Qin's Method}

We now introduce notation that will allow the states of Qin's algorithm to be studied more conveniently by matrix operations.

For this purpose, we define a variation of the state, called the {\sl s-state} (signed state), by
\[
\widehat{{\cal X}_k}
=
\begin{pmatrix}
	x_{11}^{(k)}&-x_{12}^{(k)}\\
	x_{21}^{(k)}&x_{22}^{(k)}
\end{pmatrix}.
\]
We denote its two row vectors by $\widehat{v_1}^{(k)}$ and $\widehat{v_2}^{(k)}$, respectively:
\[
\widehat{v_1}^{(k)}
=
(x_{11}^{(k)},-x_{12}^{(k)}),
\qquad
\widehat{v_2}^{(k)}
=
(x_{21}^{(k)},x_{22}^{(k)}).
\]

We collect below several useful properties of the states, the s-states, and their row vectors.

\begin{enumerate}
	
	\item
	Every entry $x_{ij}^{(k)}$ of the state ${\cal X}_k$ is nonnegative. In particular, unlike the extended Euclidean algorithm, Qin's algorithm always returns a positive representative of the modular inverse.
	
	\item
	Starting from the initial state
$
	{\cal X}_0=
	\begin{pmatrix}
		1&a\\
		0&m
	\end{pmatrix},
$
	Qin's algorithm gives the following recurrence for the sequence $\{\widehat{{\cal X}_k}\}$:
	\begin{equation}\label{eq:2.1}
		\widehat{{\cal X}_k}
		=
		\left\{
		\begin{array}{ll}
			\begin{pmatrix}1&0\\q_k&1\end{pmatrix}
			\widehat{{\cal X}_{k-1}},
			&\mbox{if $k$ is odd},\\[3mm]
			\begin{pmatrix}1&0\\q_k&1\end{pmatrix}^{\top}
			\widehat{{\cal X}_{k-1}},
			&\mbox{if $k$ is even}.
		\end{array}
		\right.
	\end{equation}
	Indeed,
	\[
	\widehat{{\cal X}_1}
	=
	\begin{pmatrix}
		1&-a\\
		q_1&m-q_1a
	\end{pmatrix}
	=
	\begin{pmatrix}
		1&0\\
		q_1&1
	\end{pmatrix}
	\begin{pmatrix}
		1&-a\\
		0&m
	\end{pmatrix}
	=
	\begin{pmatrix}
		1&0\\
		q_1&1
	\end{pmatrix}
	\widehat{{\cal X}_0},
	\]
	and
	\[
	\widehat{{\cal X}_2}
	=
	\begin{pmatrix}
		x_{11}^{(1)}+q_2x_{21}^{(1)}
		&
		-x_{12}^{(1)}+q_2x_{22}^{(1)}
		\\
		x_{21}^{(1)}
		&
		x_{22}^{(1)}
	\end{pmatrix}
	=
	\begin{pmatrix}
		1&q_2\\
		0&1
	\end{pmatrix}
	\widehat{{\cal X}_1}.
	\]
	The general case follows in the same manner.
	
	\item
	At every step $k$ of Qin's method, we always have $\det(\widehat{{\cal X}_k})=m$, i.e.
	\begin{equation}\label{eq:2.2}
		x_{11}^{(k)}x_{22}^{(k)}+x_{12}^{(k)}x_{21}^{(k)} = m.
	\end{equation}
	This identity was proved in \cite{XuLi,Xu}. We shall refer to it as {\sl Qin's invariant}.
	
	\item
	Let
	\[
	{\cal X}_N
	=
	\begin{pmatrix}
		x_{11}^{(N)}&x_{12}^{(N)}\\
		x_{21}^{(N)}&x_{22}^{(N)}
	\end{pmatrix}
	\]
	be the final state. As noted above, $N$ is even, and
	\[
	x_{12}^{(N)}=1,
	\qquad
	x_{11}^{(N)}\equiv a^{-1}\pmod m.
	\]
	
	Apply an elementary row operation to the final s-state by multiplying its first row by $x_{22}^{(N)}$ and adding the result to the second row. Using Qin's invariant together with $x_{12}^{(N)}=1$, we obtain
	\[
	\begin{pmatrix}
		x_{11}^{(N)}&-x_{12}^{(N)}\\
		x_{21}^{(N)}&x_{22}^{(N)}
	\end{pmatrix}
	\Rightarrow
	\begin{pmatrix}
		x_{11}^{(N)}&-x_{12}^{(N)}\\
		x_{21}^{(N)}+x_{11}^{(N)}x_{22}^{(N)}&0
	\end{pmatrix}
	=
	\begin{pmatrix}
		a^{-1}&-1\\
		m&0
	\end{pmatrix}.
	\]
	This row operation is consistent with the operations used in Qin's algorithm. Thus, starting from the initial s-state
	\[
	\begin{pmatrix}
		1&-a\\
		0&m
	\end{pmatrix},
	\]
	Qin's algorithm ultimately leads to
	\[
	\begin{pmatrix}
		a^{-1}&-1\\
		m&0
	\end{pmatrix}.
	\]
	This shows that the choice of state variables in Qin's algorithm is mathematically natural and reveals a remarkable duality between the initial and final states.
	
	\item
	The entries of the states satisfy
	\begin{eqnarray}\label{eq:2.3}
		&&x_{11}^{(0)}=x_{11}^{(1)}
		<x_{11}^{(2)}=x_{11}^{(3)}
		<x_{11}^{(4)}=\cdots \notag\\
		&&x_{21}^{(0)}
		<x_{21}^{(1)}=x_{21}^{(2)}
		<x_{21}^{(3)}=x_{21}^{(4)}
		<\cdots\\
		&&x_{12}^{(0)}=x_{12}^{(1)}
		>x_{12}^{(2)}=x_{12}^{(3)}
		>x_{12}^{(4)}=\cdots \notag\\
		&&x_{22}^{(0)}
		>x_{22}^{(1)}=x_{22}^{(2)}
		>x_{22}^{(3)}=x_{22}^{(4)}
		>\cdots. \notag
	\end{eqnarray}
	Thus, as $k$ increases, the entries in the left column of ${\cal X}_k$ increase, while those in the right column decrease.
	
\end{enumerate}

\section{Continued Fractions}\label{sec:CF}

In this section, we study continued fractions through the states of Qin's algorithm. We consider rational numbers, as well as finite approximations to real numbers. Let
\[
0<\lambda<1,\qquad \lambda=\frac{a}{m},
\]
where $1<a<m$ and $\gcd(a,m)=1$. Thus, the numerator $a$ and denominator $m$ of $\lambda$ are precisely the inputs to Qin's algorithm.

Some connections between the continued fraction of $\dfrac{a}{m}$ and the state matrices of Qin's algorithm were given in \cite{XuLi}. Here we develop these connections further using the s-state matrices $\{\widehat{{\cal X}_k}\}$.

Since $\lambda<1$, its continued fraction is of the form
\[
[0; q_1,q_2,\ldots,q_N,q_{N+1}],
\]
where $N$ is the number of steps in Qin's algorithm and $q_{N+1}=x_{22}^{(N)}$.
Indeed, the last step of Qin's algorithm gives $x_{12}^{(N)}=1$, so the final division in the continued fraction expansion of $\lambda$ gives
\[
q_{N+1}
=\frac{x_{22}^{(N)}}{x_{12}^{(N)}}
=x_{22}^{(N)}.
\]

Let  $\dfrac{\alpha_k}{\beta_k}$ denote the $k$-th convergent of the continued fraction of $\lambda$. We have the following relation between the convergents and the s-states of Qin's algorithm.

\begin{thm}\label{thm:3.1}
	For $k\geq 1$,
	\begin{equation}\label{eq:3.1}
		\widehat{{\cal X}_k}
		=
		\left\{
		\begin{array}{ll}
			\begin{pmatrix}
				\beta_{k-1}&\alpha_{k-1}\\
				\beta_k&\alpha_k
			\end{pmatrix}
			\widehat{{\cal X}_0},
			&\mbox{if $k$ is odd},\\[3mm]
			\begin{pmatrix}
				\beta_k&\alpha_k\\
				\beta_{k-1}&\alpha_{k-1}
			\end{pmatrix}
			\widehat{{\cal X}_0},
			&\mbox{if $k$ is even}.
		\end{array}
		\right.
	\end{equation}
In other words,
	\begin{equation}\label{eq:3.2}
		\widehat{{\cal X}_k}
		=
		\left\{
		\begin{array}{ll}
			\begin{pmatrix}
				\beta_{k-1}
				&
				m\beta_{k-1}\left(\dfrac{\alpha_{k-1}}{\beta_{k-1}}-\lambda\right)
				\\[2mm]
				\beta_k
				&
				m\beta_k\left(\dfrac{\alpha_k}{\beta_k}-\lambda\right)
			\end{pmatrix},
			&\mbox{if $k$ is odd},\\[6mm]
			\begin{pmatrix}
				\beta_k
				&
				m\beta_k\left(\dfrac{\alpha_k}{\beta_k}-\lambda\right)
				\\[2mm]
				\beta_{k-1}
				&
				m\beta_{k-1}\left(\dfrac{\alpha_{k-1}}{\beta_{k-1}}-\lambda\right)
			\end{pmatrix},
			&\mbox{if $k$ is even}.
		\end{array}
		\right.
	\end{equation}
\end{thm}

\begin{proof}
	Recall that the numerators $\alpha_k$ and denominators $\beta_k$ of the convergents satisfy
	\[
	\alpha_0=0,\qquad
	\alpha_1=1,\qquad
	\alpha_k=q_k\alpha_{k-1}+\alpha_{k-2}\quad(k\geq2),
	\]
	and
	\[
	\beta_0=1,\qquad
	\beta_1=q_1,\qquad
	\beta_k=q_k\beta_{k-1}+\beta_{k-2}\quad(k\geq2).
	\]
	For $k=1$,
	\[
	\widehat{{\cal X}_1}
	=
	\begin{pmatrix}1&0\\q_1&1\end{pmatrix}
	\widehat{{\cal X}_0}
	=
	\begin{pmatrix}
		\beta_0&\alpha_0\\
		\beta_1&\alpha_1
	\end{pmatrix}
	\widehat{{\cal X}_0},
	\]
	and for $k=2$,
	\[
	\begin{aligned}
		\widehat{{\cal X}_2}
		&=
		\begin{pmatrix}1&0\\q_2&1\end{pmatrix}^{\top}
		\widehat{{\cal X}_1}=
		\begin{pmatrix}1&0\\q_2&1\end{pmatrix}^{\top}
		\begin{pmatrix}1&0\\q_1&1\end{pmatrix}
		\widehat{{\cal X}_0}\\
		&=
		\begin{pmatrix}
			\beta_2&\alpha_2\\
			\beta_1&\alpha_1
		\end{pmatrix}
		\widehat{{\cal X}_0}.
	\end{aligned}
	\]
	
	The general case follows by induction. Suppose that (\ref{eq:3.1}) holds for $k-1$. If $k$ is odd, then
	\[
	\begin{aligned}
		\widehat{{\cal X}_k}
		&=
		\begin{pmatrix}1&0\\q_k&1\end{pmatrix}
		\widehat{{\cal X}_{k-1}}=
		\begin{pmatrix}1&0\\q_k&1\end{pmatrix}
		\begin{pmatrix}
			\beta_{k-1}&\alpha_{k-1}\\
			\beta_{k-2}&\alpha_{k-2}
		\end{pmatrix}
		\widehat{{\cal X}_0}\\
		&=
		\begin{pmatrix}
			\beta_{k-1}&\alpha_{k-1}\\
			\beta_k&\alpha_k
		\end{pmatrix}
		\widehat{{\cal X}_0}.
	\end{aligned}
	\]
	The case of even $k$ follows similarly.
	
	Finally for even $k$, note that $	\widehat{{\cal X}_0}
	=
	\begin{pmatrix}
		1&-a\\
		0&m
	\end{pmatrix}$,  we have
	\[
		\widehat{{\cal X}_k}=
		\begin{pmatrix}
			\beta_k&\alpha_k\\
			\beta_{k-1}&\alpha_{k-1}
		\end{pmatrix}
		\widehat{{\cal X}_0}=
		\begin{pmatrix}
			\beta_k&m\alpha_k-a\beta_k\\
			\beta_{k-1}&m\alpha_{k-1}-a\beta_{k-1}
		\end{pmatrix}.
	\]
	Since $\lambda=\dfrac{a}m$, this gives (\ref{eq:3.2}). The odd case is identical.
\end{proof}

Theorem \ref{thm:3.1} has several useful consequences. Some give direct matrix formulations of standard facts about continued fractions, while others provide relations that will be useful later.

\medskip
\noindent{\bf Remark 1.}
Equation (\ref{eq:3.1}) gives a direct matrix relation between the convergents $\dfrac{\alpha_k}{\beta_k}$ and the quotients $q_k$ occurring in Qin's algorithm. If $k$ is odd,
\begin{equation}\label{mat:eq1}
	\begin{pmatrix}
		\beta_{k-1}&\alpha_{k-1}\\
		\beta_k&\alpha_k
	\end{pmatrix}
	=
	\begin{pmatrix}1&0\\q_k&1\end{pmatrix}
	\begin{pmatrix}1&0\\q_{k-1}&1\end{pmatrix}^{\top}
	\cdots
	\begin{pmatrix}1&0\\q_1&1\end{pmatrix},
\end{equation}
while if $k$ is even,
\begin{equation}\label{mat:eq2}
	\begin{pmatrix}
		\beta_k&\alpha_k\\
		\beta_{k-1}&\alpha_{k-1}
	\end{pmatrix}
	=
	\begin{pmatrix}1&0\\q_k&1\end{pmatrix}^{\top}
	\begin{pmatrix}1&0\\q_{k-1}&1\end{pmatrix}
	\cdots
	\begin{pmatrix}1&0\\q_1&1\end{pmatrix}.
\end{equation}
Since every matrix on the right-hand sides of (\ref{mat:eq1}) and (\ref{mat:eq2}) has determinant $1$, the standard identity
\begin{equation}\label{det:eq1}
	\alpha_k\beta_{k-1}-\alpha_{k-1}\beta_k
	=
	(-1)^{k-1}
\end{equation}
follows immediately.

\medskip
\noindent{\bf Remark 2.}
Equation (\ref{eq:3.2}) also gives several standard properties of continued fractions directly from the s-states of Qin's algorithm.

\begin{enumerate}
	
	\item
	Suppose that $k$ is odd. By (\ref{eq:3.2}),
	\[
	m\beta_{k-1}
	\left(
	\frac{\alpha_{k-1}}{\beta_{k-1}}-\lambda
	\right)
	=
	-x_{12}^{(k)}<0
	\]
	and
	\[
	m\beta_k
	\left(
	\frac{\alpha_k}{\beta_k}-\lambda
	\right)
	=
	x_{22}^{(k)}>0.
	\]
	Hence
	\[
	\frac{\alpha_{k-1}}{\beta_{k-1}}
	<
	\lambda
	<
	\frac{\alpha_k}{\beta_k}.
	\]
	
	Now consider the two consecutive s-states 
	\[
	\widehat{{\cal X}_k}
	=
	\begin{pmatrix}
		\beta_{k-1}
		&
		m\beta_{k-1}
		\left(\dfrac{\alpha_{k-1}}{\beta_{k-1}}-\lambda\right)
		\\[2mm]
		\beta_k
		&
		m\beta_k
		\left(\dfrac{\alpha_k}{\beta_k}-\lambda\right)
	\end{pmatrix}
, \ 
	\widehat{{\cal X}_{k+1}}
	=
	\begin{pmatrix}
		\beta_{k+1}
		&
		m\beta_{k+1}
		\left(\dfrac{\alpha_{k+1}}{\beta_{k+1}}-\lambda\right)
		\\[2mm]
		\beta_k
		&
		m\beta_k
		\left(\dfrac{\alpha_k}{\beta_k}-\lambda\right)
	\end{pmatrix}.
	\]
	By Qin's invariant, $
	\det(\widehat{{\cal X}_k})
	=
	\det(\widehat{{\cal X}_{k+1}})
	=
	m$, it follows that
	\[
	m\beta_{k-1}\beta_k
	\left(
	\frac{\alpha_k}{\beta_k}
	-
	\frac{\alpha_{k-1}}{\beta_{k-1}}
	\right)
	=
	m\beta_{k+1}\beta_k
	\left(
	\frac{\alpha_k}{\beta_k}
	-
	\frac{\alpha_{k+1}}{\beta_{k+1}}
	\right).
	\]
	Since $\beta_{k-1}<\beta_{k+1}$, we obtain
	\[
	\frac{\alpha_{k-1}}{\beta_{k-1}}
	<
	\frac{\alpha_{k+1}}{\beta_{k+1}}.
	\]
	The even case is similar. Thus we recover the standard alternating property of the convergents:
	\[
	\frac{\alpha_2}{\beta_2}
	<
	\frac{\alpha_4}{\beta_4}
	<
	\cdots
	\leq
	\lambda
	<
	\cdots
	<
	\frac{\alpha_3}{\beta_3}
	<
	\frac{\alpha_1}{\beta_1}.
	\]
	
	\item
	Equation (\ref{eq:3.2}) shows that the approximation errors of consecutive convergents are directly encoded in the entries of the s-state. Together with Qin's invariant, it gives
	\[
	m\beta_k\beta_{k-1}
	\left(
	\left|
	\frac{\alpha_{k-1}}{\beta_{k-1}}-\lambda
	\right|
	+
	\left|
	\frac{\alpha_k}{\beta_k}-\lambda
	\right|
	\right)
	=
	m.
	\]
	Therefore,
	\[
	\left|
	\frac{\alpha_{k-1}}{\beta_{k-1}}-\lambda
	\right|
	+
	\left|
	\frac{\alpha_k}{\beta_k}-\lambda
	\right|
	=
	\frac{1}{\beta_{k-1}\beta_k}
	<
	\frac{1}{\beta_{k-1}^2}.
	\]

This identity gives a direct error estimate of two consecutive convergents. The following consequences are classical in the theory of continued fractions, although here they follow directly from Qin's invariant.
First,
	\[
	\left|
	\frac{\alpha_{k-1}}{\beta_{k-1}}-\lambda
	\right|
	<
	\frac{1}{\beta_{k-1}^2}.
	\]
	
	Moreover, since
	\[
	\frac{1}{2\beta_{k-1}^2}
	+
	\frac{1}{2\beta_k^2}
	>
	\frac{1}{\beta_{k-1}\beta_k},
	\]
	at least one of the following two inequalities must hold:
	\[
	\left|
	\frac{\alpha_{k-1}}{\beta_{k-1}}-\lambda
	\right|
	<
	\frac{1}{2\beta_{k-1}^2},
	\qquad
	\left|
	\frac{\alpha_k}{\beta_k}-\lambda
	\right|
	<
	\frac{1}{2\beta_k^2}.
	\]
	
\end{enumerate}

\medskip
\noindent{\bf Remark 3.}
The state interpretation also gives a simple observation concerning the penultimate convergent of $\lambda=\dfrac{a}{m}$. As mentioned at the beginning of this section, if $N$ is the number of steps in Qin's algorithm, then
\[
\lambda=[0;q_1,q_2,\ldots,q_N,q_{N+1}],
\]
where $q_{N+1}=x_{22}^{(N)}$. Since $x_{12}^{(N)}=1$, Qin's algorithm gives
$
\beta_N=x_{11}^{(N)}=a^{-1}$,
where $a^{-1}$ denotes the least positive inverse of $a$ modulo $m$. Moreover, since the final convergent is $\dfrac{a}{m}$, its denominator is $\beta_{N+1}=m$. Thus,
\[
\beta_N=a^{-1},\qquad \beta_{N+1}=m.
\]

\section{Two-dimensional Lattices}\label{sec:Latt}

For fixed integers $a,m$ with $1<a<m$ and $\gcd(a,m)=1$, define the two-dimensional lattice
\[
\Lambda(a,m):=\{(x,y)\in\mathbb Z^2: ax+y\equiv0\pmod m\}.
\]
This is a common example of a two-dimensional lattice and has been used in various applications; see, for example, \cite{glv}. We first show that the s-states generated by Qin's algorithm give bases of $\Lambda(a,m)$.

\begin{thm}
	Every s-state $\widehat{{\cal X}_k}$ is a basis matrix of $\Lambda(a,m)$. In particular, the volume of $\Lambda(a,m)$ is $m$.
\end{thm}

\begin{proof}
	The rows of
	\[
	\widehat{{\cal X}_0}
	=
	\begin{pmatrix}
		1&-a\\
		0&m
	\end{pmatrix}
	\]
	form a basis of $\Lambda(a,m)$. Indeed, for any $(x,y)\in\Lambda(a,m)$, there is an integer $t$ such that
	\[
	ax+y=tm.
	\]
	Hence
	\[
	(x,y)=x(1,-a)+t(0,m).
	\]
	
	For every $k$, the rows of $\widehat{{\cal X}_k}$ also form a basis of $\Lambda(a,m)$, since $\widehat{{\cal X}_k}$ is obtained from $\widehat{{\cal X}_0}$ by a sequence of unimodular row transformations of the form
	\[
	\begin{pmatrix}1&q\\0&1\end{pmatrix}
	\qquad\mbox{or}\qquad
	\begin{pmatrix}1&0\\q&1\end{pmatrix}.
	\]
	By Qin's invariant,
	\[
	\det\widehat{{\cal X}_k}=m.
	\]
	Therefore, the volume of $\Lambda(a,m)$ is $m$.
\end{proof}

We also note that
\[
\begin{pmatrix}
	a^{-1} &-1\\
	m&0
\end{pmatrix}
\]
is a basis matrix of $\Lambda(a,m)$. Indeed, as shown in Section \ref{sec:QinAlg},
\[
\begin{pmatrix}
	a^{-1} &-1\\
	m&0
\end{pmatrix}
=
\begin{pmatrix}
	1&0\\
	x_{22}^{(N)}&1
\end{pmatrix}
\widehat{{\cal X}_N}.
\]

\subsection{Shortest Vectors of $\Lambda(a,m)$}

We next investigate whether a shortest vector of $\Lambda(a,m)$ can be obtained directly from an s-state of Qin's algorithm.

We begin with two simple cases. As observed above,
\[
\begin{pmatrix}1&-a\\0&m\end{pmatrix}
\qquad\mbox{and}\qquad
\begin{pmatrix}a^{-1}&-1\\m&0\end{pmatrix}
\]
are both basis matrices of $\Lambda(a,m)$. If either $a$ or $a^{-1}$ is sufficiently small, a shortest vector can be identified immediately.

\begin{prop}\label{prop:trivialSV}
	\begin{enumerate}
		\item If $a^2<m$, then $(1,-a)$ is a shortest vector of $\Lambda(a,m)$.
		\item If $(a^{-1})^2<m$, then $(a^{-1},-1)$ is a shortest vector of $\Lambda(a,m)$.
	\end{enumerate}
\end{prop}

\begin{proof}
	We prove the first statement. Suppose, to the contrary, that there is a nonzero vector
	\[
	v=(k_1,k_2m-k_1a)\in\Lambda(a,m)
	\]
	such that
	\[
	\|v\|<\sqrt{a^2+1}.
	\]
	Since $\|v\|^2$ is an integer, this implies
	\[
	\|v\|^2\leq a^2.
	\]
	Replacing $v$ by $-v$ if necessary, we may assume $k_1>0$. Since
	\[
	(k_2m-k_1a)^2\leq a^2,
	\]
	we must have $k_2>0$. Also $k_1\leq a$, and therefore
	\[
	k_2m-k_1a
	\geq k_2m-a^2
	=
	(k_2-1)m+(m-a^2).
	\]
	Since $m-a^2>0$, this forces $k_2=1$.
	
	We are thus left with
	\[
	k_1^2+(m-k_1a)^2\leq a^2.
	\]
	If $k_1=a$, then this inequality forces $m-a^2=0$, contradicting $a^2<m$. If $k_1<a$, then
	\[
	m-k_1a
	\geq m-(a-1)a
	=
	a+(m-a^2)
	>
	a,
	\]
	again contradicting the inequality above. Hence $(1,-a)$ is a shortest vector.
	
	The second statement follows by the same argument, using the basis $
	\begin{pmatrix}
		a^{-1}&-1\\
		m&0
	\end{pmatrix}$.
\end{proof}

We now show that a shortest vector can always be obtained from the rows of an s-state, possibly by taking their sum or difference.

\begin{thm}\label{thm:sv}
	There exists an s-state
$
	\widehat{{\cal X}_k}
	=
	\begin{pmatrix}
		\widehat v_1^{(k)}\\
		\widehat v_2^{(k)}
	\end{pmatrix}
$
	such that
	\[
	\{
	\widehat v_1^{(k)},
	\widehat v_2^{(k)},
	\widehat v_1^{(k)}+\widehat v_2^{(k)},
	\widehat v_1^{(k)}-\widehat v_2^{(k)}
	\}
	\]
	contains a shortest vector of $\Lambda(a,m)$.
\end{thm}

The following lemma will be used in the proof. The first statement is a standard result on continued fractions. The second is due to Lang \cite[Chapter 1, Theorem 10]{lang} (see also \cite{Wor}). It concerns the intermediate fractions studied by Khinchin \cite{khin}.

\begin{lem}\label{lem:lang}
	Let $\lambda\in\mathbb R$, and let
$
	\big\{\dfrac{\alpha_j}{\beta_j}:j=0,1,\ldots\big\}
$
	be the sequence of convergents of the continued fraction expansion of $\lambda$.
	\begin{enumerate}
		\item If integers $u,v$ satisfy
		\[
		\left|\lambda-\frac uv\right|
		\leq\frac{1}{2v^2},
		\]
		then, for some $j$,
		\[
		\frac uv=\frac{\alpha_j}{\beta_j}.
		\]
		
		\item If integers $u,v$ satisfy
		\[
		\left|\lambda-\frac uv\right|
		\leq\frac{1}{v^2},
		\]
		then, for some $j$,
		\[
		\frac uv=\frac{\alpha_j}{\beta_j}
		\]
		or
		\[
		\frac uv
		=
		\frac{\alpha_j\pm\alpha_{j-1}}
		{\beta_j\pm\beta_{j-1}}.
		\]
	\end{enumerate}
\end{lem}

\begin{proof}[Proof of Theorem \ref{thm:sv}]
	Let $(x_0,y_0)$ be a nonzero shortest vector of $\Lambda(a,m)$. It is easy to see that $x_0\neq 0$, otherwise $y_0$ would be a multiple of $m$, and this yields a contradiction as 
	\[
	m^2\le y_0^2\le 1+a^2<m^2.
	\]	
	Replacing $x_0$ by its negative if necessary, we may assume $x_0>0$.
	
	Note that $\beta_0=1$ and $\beta_{N+1}=m$, we may choose $k>0$ such that
	\[
	\beta_{k-1}\leq x_0<\beta_k.
	\]
	By Theorem \ref{thm:3.1}, the vectors
	\[
	(\beta_k,m\alpha_k-a\beta_k)
	\quad\mbox{and}\quad
	(\beta_{k-1},m\alpha_{k-1}-a\beta_{k-1})
	\]
	are the two rows of $\widehat{{\cal X}_k}$. Since $(x_0,y_0)$ is a shortest vector,
	\[
	x_0^2+y_0^2
	\leq
	\beta_{k-1}^2+
	(m\alpha_{k-1}-a\beta_{k-1})^2.
	\]
	Together with $\beta_{k-1}\leq x_0$, this gives
	\[
	|y_0|
	\leq
	|m\alpha_{k-1}-a\beta_{k-1}|.
	\]
	Using Qin's invariant,
	\[
	m
	=
	\beta_k|m\alpha_{k-1}-a\beta_{k-1}|
	+
	\beta_{k-1}|m\alpha_k-a\beta_k|,
	\]
	and hence
	\[
	|x_0y_0|
	<
	\beta_k|m\alpha_{k-1}-a\beta_{k-1}|
	<
	m.
	\]
	
	Since $(x_0,y_0)\in\Lambda(a,m), \
	\dfrac{y_0+ax_0}{m}$	is an integer. Therefore,
	\begin{equation}\label{eq:contF}
		\left|
		\frac am-
		\frac{(y_0+ax_0)/m}{x_0}
		\right|
		=
		\frac{|y_0|}{mx_0}
		<
		\frac{1}{x_0^2}.
	\end{equation}
	
	Let $
	d=\gcd\left(x_0,\dfrac{y_0+ax_0}{m}\right)$. 
If $d>1$, then (\ref{eq:contF}) gives
	\[
	\left|
	\frac am-
	\frac{(y_0+ax_0)/(dm)}{x_0/d}
	\right|
	<
	\frac{1}{d^2(x_0/d)^2}
	<
	\frac{1}{2(x_0/d)^2}.
	\]
	By the first part of Lemma \ref{lem:lang}, there is a $j$ such that
	\[
	\frac{x_0}{d}=\beta_j,
	\qquad
	\frac{y_0+ax_0}{dm}=\alpha_j.
	\]
	Consequently,
	\[
	x_0^2+y_0^2
	=
	d^2\bigl(\beta_j^2+(m\alpha_j-a\beta_j)^2\bigr)
	>
	\beta_j^2+(m\alpha_j-a\beta_j)^2.
	\]
	But
	\[
	(\beta_j,m\alpha_j-a\beta_j)\in\Lambda(a,m),
	\]
	contradicting the minimality of $(x_0,y_0)$. Hence $d=1$.
	
	Applying the second part of Lemma \ref{lem:lang} to (\ref{eq:contF}), we obtain either
	\[
	\frac{(y_0+ax_0)/m}{x_0}
	=
	\frac{\alpha_j}{\beta_j}, \quad \mbox{  or } \quad
	\frac{(y_0+ax_0)/m}{x_0}
	=
	\frac{\alpha_j\pm\alpha_{j-1}}
	{\beta_j\pm\beta_{j-1}}
	\]
	for some $j$.
	
	In the former case,
	\[
	x_0=\beta_j,
	\qquad
	\frac{y_0+ax_0}{m}=\alpha_j.
	\]
	Since
	\[
	\beta_{k-1}\leq x_0<\beta_k,
	\]
	we have $j=k-1$, and therefore
	\[
	(x_0,y_0)
	=
	(\beta_{k-1},m\alpha_{k-1}-a\beta_{k-1}),
	\]
	which is a row of $\widehat{{\cal X}_k}$.
	
	In the latter case,  since	$(\alpha_j\pm\alpha_{j-1})\beta_{j-1}
	-
	(\beta_j\pm\beta_{j-1})\alpha_{j-1}
	=
	\pm1$, the fractions
	\[
	\frac{\alpha_j\pm\alpha_{j-1}}
	{\beta_j\pm\beta_{j-1}}
	\]
	are reduced.
	Since $d=1$, we must therefore have either
	\[
	x_0=\beta_j+\beta_{j-1},
	\qquad
	\frac{y_0+ax_0}{m}
	=
	\alpha_j+\alpha_{j-1},
	\]
	or
	\[
	x_0=\beta_j-\beta_{j-1},
	\qquad
	\frac{y_0+ax_0}{m}
	=
	\alpha_j-\alpha_{j-1}.
	\]
	Accordingly,
	\[
	(x_0,y_0)
	=
	(\beta_j,m\alpha_j-a\beta_j)
	+
	(\beta_{j-1},m\alpha_{j-1}-a\beta_{j-1}),
	\]
	or
	\[
	(x_0,y_0)
	=
	(\beta_j,m\alpha_j-a\beta_j)
	-
	(\beta_{j-1},m\alpha_{j-1}-a\beta_{j-1}).
	\]
	Thus $(x_0,y_0)$ is either a row, the sum of the two rows, or the difference of the two rows of an s-state.
\end{proof}

For finding a shortest vector, one may run Qin's algorithm and examine the s-states according to Theorem \ref{thm:sv}. We next give a more precise way to locate a state containing a shortest vector.

For an s-state $\widehat{{\cal X}_k}$, let
\[
{\cal I}_k
=
\langle\widehat v_1^{(k)},\widehat v_2^{(k)}\rangle
=
x_{11}^{(k)}x_{21}^{(k)}
-
x_{12}^{(k)}x_{22}^{(k)}
\]
denote the inner product of its two row vectors.

Qin's procedure gives the following recurrence, which in particular shows that the sequence $\{{\cal I}_k\}$ is strictly increasing.

\begin{prop}\label{prop:Inner}
	For $k\geq1$,
	\[
	{\cal I}_k
	=
	\left\{
	\begin{array}{ll}
		{\cal I}_{k-1}
		+
		q_k\|\widehat v_1^{(k)}\|^2,
		&\mbox{if $k$ is odd},\\[2mm]
		{\cal I}_{k-1}
		+
		q_k\|\widehat v_2^{(k)}\|^2,
		&\mbox{if $k$ is even}.
	\end{array}
	\right.
	\]
\end{prop}

\begin{proof}
	Suppose first that $k$ is odd. Then
	\[
	x_{11}^{(k)}=x_{11}^{(k-1)},\qquad
	x_{12}^{(k)}=x_{12}^{(k-1)},
	\]
	and
	\[
	x_{21}^{(k)}
	=
	x_{21}^{(k-1)}+q_kx_{11}^{(k-1)},
	\qquad
	x_{22}^{(k)}
	=
	x_{22}^{(k-1)}-q_kx_{12}^{(k-1)}.
	\]
	Hence
	\[
	\begin{aligned}
		{\cal I}_k
		&=
		x_{11}^{(k)}x_{21}^{(k)}
		-
		x_{12}^{(k)}x_{22}^{(k)}\\
		&=
		x_{11}^{(k-1)}
		\bigl(x_{21}^{(k-1)}+q_kx_{11}^{(k-1)}\bigr)
		-
		x_{12}^{(k-1)}
		\bigl(x_{22}^{(k-1)}-q_kx_{12}^{(k-1)}\bigr)\\
		&=
		{\cal I}_{k-1}
		+
		q_k\|\widehat v_1^{(k-1)}\|^2\\
		&=
		{\cal I}_{k-1}
		+
		q_k\|\widehat v_1^{(k)}\|^2.
	\end{aligned}
	\]
	
	If $k$ is even, then
	\[
	x_{11}^{(k)}
	=
	x_{11}^{(k-1)}+q_kx_{21}^{(k-1)},
	\qquad
	x_{12}^{(k)}
	=
	x_{12}^{(k-1)}-q_kx_{22}^{(k-1)},
	\]
	while
	\[
	x_{21}^{(k)}=x_{21}^{(k-1)},
	\qquad
	x_{22}^{(k)}=x_{22}^{(k-1)}.
	\]
	Therefore,
	\[
	\begin{aligned}
		{\cal I}_k
		&=
		x_{11}^{(k)}x_{21}^{(k)}
		-
		x_{12}^{(k)}x_{22}^{(k)}\\
		&=
		x_{21}^{(k-1)}
		\bigl(x_{11}^{(k-1)}+q_kx_{21}^{(k-1)}\bigr)
		-
		x_{22}^{(k-1)}
		\bigl(x_{12}^{(k-1)}-q_kx_{22}^{(k-1)}\bigr)\\
		&=
		{\cal I}_{k-1}
		+
		q_k\|\widehat v_2^{(k-1)}\|^2\\
		&=
		{\cal I}_{k-1}
		+
		q_k\|\widehat v_2^{(k)}\|^2.
	\end{aligned}
	\]
\end{proof}

We now consider the nontrivial case in which $(a^{-1},-1)$ is not a shortest vector of $\Lambda(a,m)$. By Proposition \ref{prop:Inner}, $\{{\cal I}_k\}$ is strictly increasing, and
\[
{\cal I}_0=-am<0.
\]
We next show that
\[
{\cal I}_N>0.
\]

Suppose, to the contrary, that ${\cal I}_N\leq0$. 

Since $N$ is even, the final s-state has the form
\[
\widehat{{\cal X}_N}
=
\begin{pmatrix}
	x_{11}^{(N)}&-1\\
	x_{21}^{(N-1)}&x_{22}^{(N-1)}
\end{pmatrix}
=
\begin{pmatrix}
	\beta_N&-1\\
	\beta_{N-1}&m\alpha_{N-1}-a\beta_{N-1} .
\end{pmatrix}
\]
Then the  assumption that  ${\cal I}_N\leq0$ is equivalent to
\[
\beta_N\beta_{N-1}
\leq
m\alpha_{N-1}-a\beta_{N-1}.
\]
On the other hand,
\[
\frac{\alpha_{N-1}}{\beta_{N-1}}-\frac am
\leq
\frac{1}{\beta_{N-1}\beta_N},
\]
and therefore
\[
m\alpha_{N-1}-a\beta_{N-1}
\leq
\frac{m}{\beta_N}.
\]
Combining the two inequalities gives
\[
m\geq\beta_N^2\beta_{N-1}.
\]
Since $\beta_{N-1}\geq1$ and $\beta_N = a^{-1}$, we have
\[
(a^{-1})^2 \leq m.
\]
The equality cannot occur as $\gcd(a^{-1}, m)=1$. Hence
\[
(a^{-1})^2<m.
\]
By proposition \ref{prop:trivialSV}, $(a^{-1},-1)$ would then be a shortest vector, a contradiction. Thus
\[
{\cal I}_N>0.
\]

Consequently, there is a unique $k_0$ such that
\[
{\cal I}_{k_0}<0<{\cal I}_{k_0+1}.
\]
It follows that the minimum absolute value of the inner products is attained at either ${\cal I}_{k_0}$ or ${\cal I}_{k_0+1}$. This suggests that the two corresponding s-states may contain short lattice vectors.

Our computational experiments revealed two interesting phenomena. First, the step $k_0$ generally occurs near the middle of the execution of Qin's algorithm. Second, a shortest vector of $\Lambda(a,m)$ always appeared as a row of one of the two consecutive s-states $\widehat{{\cal X}_{k_0}}$ and $\widehat{{\cal X}_{k_0+1}}$.
These observations led us to conjecture, in an earlier version of this paper \cite{WuXu} (see also \cite{Xu-book}), that this property holds in general.

The conjecture is closely related to Theorem \ref{thm:sv}, which establishes that a shortest vector can always be obtained from the two row vectors of an s-state, either directly or by taking their sum or difference. The conjecture asserts a stronger property: the sum and difference are unnecessary, and the relevant s-state can be located by the sign change of the inner product. Thus, it provides both a sharper characterization of shortest vectors and a more precise way of finding them.

The conjecture has recently been proved by Zou \cite{Zou}. We state his result below.

\begin{thm}[Zou \cite{Zou}]\label{thm:zou}
	Let $k_0$ satisfy
	\[
	{\cal I}_{k_0}<0<{\cal I}_{k_0+1}.
	\]
	Then at least one of the two s-states
	\[
	\widehat{{\cal X}_{k_0}}
	\qquad\mbox{and}\qquad
	\widehat{{\cal X}_{k_0+1}}
	\]
	contains a shortest vector of $\Lambda(a,m)$ as one of its rows.
\end{thm}

A key ingredient in Zou's proof is Theorem \ref{thm:sv}. Starting from the restricted set of candidates provided by that theorem, his argument shows that the sum and difference of the two row vectors can be eliminated from consideration. The monotonicity of the inner products then allows the search to be confined to the two consecutive s-states surrounding the sign change.

Consequently, a shortest vector of $\Lambda(a,m)$ always occurs as a row of an s-state generated by Qin's algorithm. This strengthens Theorem \ref{thm:sv} in two respects: it eliminates the need to consider sums and differences of row vectors, and it identifies the precise location in the sequence of s-states where a shortest vector can be found. In particular, once the sign change is detected, it suffices to examine the four row vectors of the two corresponding s-states.

We conclude the section with an example.

\medskip
\noindent{\bf Example.}
Consider the lattice $\Lambda(38887,41130)$. The s-states generated by Qin's algorithm, together with their inner products, are shown below.

\scriptsize{
	\begin{center}
		\begin{tabular}{|l|l|l|}
			\hline
			$k$ & $\widehat{{\cal X}_k}$ & ${\cal I}_k$\\
			\hline
			0 & $\begin{pmatrix}1&-38887\\0&41130\end{pmatrix}$ & $-1599422310$\\
			1 & $\begin{pmatrix}1&-38887\\1&2243\end{pmatrix}$ & $-87223540$\\
			2 & $\begin{pmatrix}18&-756\\1&2243\end{pmatrix}$ & $-1695690$\\
			3 & $\begin{pmatrix}18&-756\\37&731\end{pmatrix}$ & $-551970$\\
			4 & $\begin{pmatrix}55&-25\\37&731\end{pmatrix}$ & $-16240$\\
			5 & $\begin{pmatrix}55&-25\\1632&6\end{pmatrix}$ & $89610$\\
			6 & $\begin{pmatrix}6583&-1\\1632&6\end{pmatrix}$ & $10743450$\\
			\hline
		\end{tabular}
	\end{center}
}
\normalsize{}

Here
\[
{\cal I}_4<0<{\cal I}_5.
\]
A shortest vector of $\Lambda(38887,41130)$ is
\[
(55,-25),
\]
which occurs as a row of both $\widehat{{\cal X}_4}$ and $\widehat{{\cal X}_5}$.

The second shortest vector is
\[
(257,631)=(4,1)\widehat{{\cal X}_4}.
\]
Unlike the shortest vector, it does not occur as a row of any s-state $\widehat{{\cal X}_k}$.

\section*{Acknowledgement}
We thank Professors Eric Bach and Xiaoyun Wang for their valuable suggestions and comments that helped improve this paper.
This work is partially supported by the National Natural Science Foundation of China (No. 12271306).

\end{document}